\documentclass[a4paper,10pt]{article}
\usepackage[english,french]{babel}
\usepackage[T1]{fontenc}
\usepackage[utf8]{inputenc}
\usepackage{amsmath}
\usepackage{array}
\usepackage{amsthm}
\usepackage{amssymb}
\input xy
\xyoption{all}
\usepackage{hyperref}
\usepackage{stmaryrd}

\newcommand{\A}{{\mathcal{A}}}

\newcommand{\C}{{\mathcal{C}}}
\newcommand{\F}{{\mathcal{F}}}

\newcommand{\G}{{\mathrm{G}}}

\newcommand{\Hom}{\mathrm{Hom}}
\newcommand{\End}{\mathrm{End}}
\newcommand{\col}{{\rm colim}\,}
\newcommand{\Md}{\text{-}\mathbf{Mod}}
\newcommand{\FF}{\mathbb{F}}
\newcommand{\GL}{\operatorname{GL}}
\newcommand{\op}{\mathrm{op}}
\newcommand{\lgr}{\ell}
\newcommand{\surj}{\operatorname{Surj}}
\newcommand{\zcyc}{\mathbb{Z}^{\operatorname{cycl}}}
\newcommand{\tr}{{\operatorname{tr}}}
\newcommand{\tor}{{\operatorname{tor}}}
\title{Note sur la linéarisation des groupes abéliens finis}

\author{Aur\'elien DJAMENT\thanks{CNRS, laboratoire Analyse, Géométrie et Applications (UMR7539), Institut Galilée, 99 avenue Jean-Baptiste Clément
93430 VILLETANEUSE,
FRANCE, djament@math.cnrs.fr.}}

\newtheorem{thm}{Th\'eor\`eme}[subsection]
\newtheorem{pr}[thm]{Proposition}
\newtheorem{cor}[thm]{Corollaire}
\newtheorem{lm}[thm]{Lemme}

\theoremstyle{definition}
\newtheorem{defi}[thm]{D\'efinition}
\newtheorem{nota}[thm]{Notation}
\newtheorem{hyp}[thm]{Hypoth\`ese}

\theoremstyle{remark}
\newtheorem{rem}[thm]{Remarque}

\begin{document}

\maketitle

\begin{abstract}
Si $K$ est un corps commutatif \textit{possédant assez de racines de l'unité} et $V$ un groupe abélien fini, la $K$-algèbre $K[V]$ du groupe $V$ est semi-simple déployée, de sorte que le morphisme canonique $K[V]\to K^{V^\sharp}$, où $V^\sharp$ désigne le groupe dual de $V$ (qu'on peut voir comme $\Hom(V,K^\times)$), est un isomorphisme de $K$-algèbres. Si l'on supprime l'hypothèse que $K$ contient assez de racines de l'unité, on peut en déduire facilement (en utilisant une extension des scalaires et Krull-Schmidt) qu'il demeure un isomorphisme, cette fois-ci seulement $K$-linéaire, $K[V]\xrightarrow{\simeq} K^{V^\sharp}$ naturel en le groupe $V$ \textit{si l'on impose à l'avance que $V$ est annulé par un entier non nul fixé} \cite[prop.~3.3]{DG}. La question de savoir s'il existe un tel isomorphisme naturel en le groupe abélien $V$, sans d'autre restriction que $V$ fini et d'ordre inversible dans $K$, est moins évidente ; nous la résolvons ici positivement, dans un cas un peu plus général ($K$ étant un anneau commutatif quelconque), en utilisant des sommes de Gau\ss. Nous explorons également quelques questions fonctorielles reliées.
\end{abstract}

{\selectlanguage{english}
{\begin{abstract}
If $K$ is a field \textit{with enough roots of unity} and $V$ an abelian group, the $K$-algebra $K[V]$ of the group $V$ is split semisimple, so that the canonical morphism $K[V]\to K^{V^\sharp}$, where $V^\sharp$ denotes the dual group of $V$ (which may be seen as $\Hom(V,K^\times)$), is an isomorphism of $K$-algebras. If one removes the assumption that $K$ has enough roots of unity, one can easily deduce from it (by using base change and Krull-Schmidt) that it remains a $K$-linear isomorphism  $K[V]\xrightarrow{\simeq} K^{V^\sharp}$ natural in the group $V$ \textit{if one restricts to finite groups $V$ cancelled by a fixed nonzero integer} \cite[prop.~3.3]{DG}. The question of whether such an isomorphism, natural in the abelian group $V$, still exists without any other restriction than $V$ is finite and its order is invertible in $K$, is less obvious; we solve it positively, in a somewhat more general setting ($K$ being any commutative ring), by using Gau\ss\ sums. We also explore some related functorial questions.
\end{abstract}
}}

\bigskip

\noindent
{\em Mots-clefs : } catégories de foncteurs, catégories additives, linéarisation, inversion de Fourier, groupes abéliens finis, dualité, sommes de Gau\ss.

\medskip

\noindent
{\em Classification MSC 2020 : } 18A25, 18E05, 20K01 ; 11L05.

%\tableofcontents

\bigskip

Dans tout ce texte, la lettre $k$ désigne un anneau commutatif non nul. On note $k^\times$ le groupe des éléments inversibles de $k$.

La lettre  $p$ désigne un nombre premier. On note $\mathbb{Z}/p^\infty$ le groupe abélien divisible $\underset{n\in\mathbb{N}}{\col}\mathbb{Z}/p^n$.

Si $E$ est un ensemble, on note $|E|$ son cardinal, et $k[E]$ le $k$-module libre sur $E$ ; on définit ainsi un foncteur $k[-]$ de la catégorie $\mathbf{Ens}$ des ensembles vers la catégorie $k\Md$ des $k$-modules. On note également $k^E$ le $k$-module des fonctions ensemblistes de $E$ vers $k$ ; on obtient ainsi un foncteur $k^{(-)} : \mathbf{Ens}^{\op}\to k\Md$. Le $k$-module $k^E$ possède également une structure naturelle de $k$-algèbre (produit de copies de $k$).

Si $M$ est un monoïde, le $k$-module $k[M]$ est muni de la structure de $k$-algèbre induite par la loi de $M$ (algèbre de monoïde).

\section{Rappels arithmétiques sur les groupes abéliens finis}

\subsection{Algèbres de groupe diagonalisables}

\subsubsection{Algèbres diagonalisables}
\begin{defi} On appelle $k$-algèbre \textbf{faiblement diagonalisable} toute $k$-algèbre isomorphe à l'algèbre produit $k^E$ pour un ensemble $E$. Si l'on peut choisir $E$ fini, on parle de $k$-algèbre \textbf{fortement diagonalisable}.
\end{defi}

\begin{rem}
Une $k$-algèbre faiblement diagonalisable qui est un $k$-module de type fini est fortement diagonalisable (on peut réduire modulo un idéal maximal $\mathfrak{m}$ de $k$ pour voir que l'espace vectoriel $(k/\mathfrak{m})^E$ sur $k/\mathfrak{m}$ est de dimension finie, ce qui entraîne que $E$ est fini).
\end{rem}

\begin{lm}\label{idem-stan} Soit $X$ un ensemble fini d'éléments idempotents d'un anneau commutatif $A$. Il existe un ensemble complet d'idempotents $T$ de $A$ (i.e. un ensemble d'idempotents non nuls de $A$ tel que $e.e'=0$ pour $e, e\in E'$ distincts et que $\sum_{e\in T}e=1$) tel que tout élément de $X$ soit somme d'éléments deux à deux distincts de $T$.
\end{lm}

\begin{proof} Ce résultat élémentaire constitue une conséquence classique de l'inversion de Möbius, voir par exemple \cite[th.~3.9.2]{Stan}.
\end{proof}

\begin{lm}\label{lm-idem_diag} Soit $A$ une $k$-algèbre. On suppose que $A$ est un $k$-module libre et qu'il existe des idempotents $e_i$ ($1\le i\le n$) de $k$ tels que $A\simeq\prod_{i=1}^n e_i.k$ comme $k$-algèbre. Alors $A$ est une $k$-algèbre fortement diagonalisable.
\end{lm}

\begin{proof} Soit $r$ le rang du $k$-module libre $A$. Soit par ailleurs $T$ un ensemble complet d'idempotents associé à l'ensemble $\{e_1,\dots,e_n\}$ comme dans le lemme~\ref{idem-stan}. Comme $\Big(\sum_{j=1}^m\epsilon_j\Big).k\simeq\prod_{j=1}^m (\epsilon_j.k)$ comme $k$-algèbres lorsque les $\epsilon_j$ sont des idempotents tels que $\epsilon_j.\epsilon_t=0$ pour $j\ne t$, on voit qu'il existe une fonction $\nu : T\to\mathbb{N}$ telle que $A$ soit isomorphe à l'algèbre $\prod_{\epsilon\in T}(\epsilon.k)^{\nu(\epsilon)}$. Cela implique notamment $\epsilon.A\simeq (\epsilon.k)^{\nu(\epsilon)}$ comme $k$-algèbres pour tout $\epsilon\in T$. En particulier, $\epsilon.A$ est un $\epsilon.k$-module libre de rang $\nu(\epsilon)$, d'où $\nu(\epsilon)=r$ pour tout $\epsilon\in T$. Par conséquent, on dispose d'un isomorphisme de $k$-algèbres $A\simeq\prod_{\epsilon\in T}(\epsilon.k)^r\simeq k^r$, puisque $T$ est un ensemble complet d'idempotents de $k$, d'où le lemme.
\end{proof}

\begin{pr}\label{pr-fcfd} Soient $L$ une $k$-algèbre fortement diagonalisable et $A$ une algèbre quotient de $A$. On suppose que la projection $L\twoheadrightarrow A$ est scindée comme morphisme de $L$-modules, et que $A$ est un $k$-module libre. Alors $A$ est une $k$-algèbre fortement diagonalisable.
\end{pr}

\begin{proof} L'algèbre $A$ est par hypothèse l'image d'un idempotent de la $k$-algèbre $L$. Or  $L\simeq k^n$ pour un certain $n\in\mathbb{N}$, et les idempotents de $k^n$ sont les $n$-uplets d'éléments idempotents de $k$. Ainsi, la proposition est une reformulation du lemme~\ref{lm-idem_diag}.
\end{proof}

Le résultat suivant est immédiat.
\begin{lm}\label{lm-ptd} Le produit tensoriel de deux $k$-algèbres fortement diagonalisables est fortement diagonalisable.
\end{lm}

\subsubsection{Éléments étrangers}

\begin{defi} Deux éléments d'un anneau commutatif $A$ sont dits \textbf{étrangers} si l'idéal qu'ils engendrent égale $A$.
\end{defi}

Les cinq énoncés suivants, très simples et classiques, sont laissés en exercice.

\begin{lm}\label{lm-chin} Deux éléments $a, b$ d'un anneau commutatif $A$ sont étrangers si et seulement si le morphisme d'anneaux canonique $A/(ab)\to A/(a)\times A/(b)$ est un isomorphisme.
\end{lm}

\begin{lm}\label{lmev-g} Si chacun des éléments d'une famille finie d'éléments deux à deux étrangers d'un anneau commutatif $A$ divise un élément $x$ de $A$, alors le produit des éléments de cette famille divise également $x$.
\end{lm}

\begin{lm}\label{lm-etr_prod} Soient $a, b$ et $x$ des éléments d'un anneau commutatif. Si $a$ et $x$ sont étrangers et que $b$ et $x$ sont étrangers, alors $ab$ et $x$ sont étrangers.
\end{lm}

\begin{lm}\label{lm-proder} Soient $R$ et $S$ des éléments de l'algèbre de polynômes $k[X]$ ; posons $P:=RS$. Les assertions suivantes sont équivalentes :
\begin{enumerate}
\item  $P$ et son polynôme dérivé $P'$ sont étrangers ;
\item $R$ et $S$ sont étrangers, $R$ et $R'$ sont étrangers, et $S$ et $S'$ sont étrangers.
\end{enumerate}
\end{lm}

\begin{lm}\label{lm-diffev} Soient $a$ et $b$ des éléments de $k$. Les polynômes $X-a$ et $X-b$ de $k[X]$ sont étrangers si et seulement si $a-b\in k^\times$.
\end{lm}

\begin{lm}\label{quot-diag} Soit $P$ un polynôme unitaire de $k[X]$. Les assertions suivantes sont équivalentes :
\begin{enumerate}
\item\label{qd1} la $k$-algèbre $k[X]/(P)$ est fortement diagonalisable ;
\item\label{qd2} $P$ se décompose en produit de polynômes unitaires de degré $1$ deux à deux étrangers ;
\item\label{qd3} $P$ se décompose en produit de polynômes unitaires de degré $1$ et est étranger à son polynôme dérivé.
\end{enumerate}
\end{lm}

\begin{proof} Soit $n$ le degré de $P$ : l'algèbre $A:=k[X]/(P)$ est fortement diagonalisable si et seulement si $A\simeq k^n$. Un morphisme d'algèbres $A\to k^n$ est induit par l'évaluation des polynômes en des éléments $a_1,\dots,a_n$ de $k$ annulant $P$. Un tel morphisme est un isomorphisme si et seulement si l'image de la base du $k$-module $A$ constituée des classes modulo $(P)$ des $X^i$ ($0\le i<n$), à savoir des $(a_j^i)_{1\le j\le n}$, est une base de $k^n$, autrement dit, si et seulement si le déterminant de la matrice correspondante est inversible. Ce déterminant de Vandermonde est inversible si et seulement si $a_i-a_j\in k^\times$ pour $i\ne j$. Compte-tenu des lemmes~\ref{lm-diffev} et~\ref{lmev-g}, cela établit l'équivalence entre \ref{qd1} et \ref{qd2}.

L'équivalence entre \ref{qd2} et \ref{qd3} résulte pour sa part d'une application itérée du lemme~\ref{lm-proder}.
\end{proof}

\begin{lm}\label{lm-dpolprod} Soient  $S$ et $T$ des polynômes unitaires de $k[X]$. Si la $k$-algèbre $k[X]/(ST)$ est fortement diagonalisable, alors il en est de même pour $k[X]/(S)$.
\end{lm}

\begin{proof} Supposons $k[X]/(ST)$ fortement diagonalisable. Les lemmes~\ref{quot-diag} et \ref{lm-proder} montrent que $S$ et $T$ sont étrangers, d'où un isomorphisme d'algèbres $k[X]/(ST)\simeq k[X]/(S)\times k[X]/(T)$ par le lemme~\ref{lm-chin}. Ainsi, la $k$-algèbre $k[X]/(S)$, qui est un module libre sur $k$ car $S$ est un polynôme unitaire, est un quotient de $k[X]/(ST)$ et l'application canonique $k[X]/(ST)\twoheadrightarrow k[X]/(S)$ est scindée comme morphisme de $k[X]/(ST)$-modules. La proposition~\ref{pr-fcfd} donne alors la conclusion.
\end{proof}

\begin{lm}\label{lm-inv_cycl} Soient $d>1$ un entier et $\alpha\in k$. On suppose que $d$ est inversible dans $k$ et que $\alpha$ annule le $d$-ième polynôme cyclotomique $\Phi_d$. Alors $\alpha-1\in k^\times$.
\end{lm}

\begin{proof} Comme $d$ est inversible dans $k$, $X^d-1\in k[X]$ est étranger à son polynôme dérivé $dX^{d-1}$. Du fait que $X^d-1$ est divisible par $(X-1)\Phi_d$ (car $d>1$), et donc par $(X-1)(X-\alpha)$, le lemme~\ref{lm-proder} permet d'en déduire que $X-1$ et $X-\alpha$ sont étrangers, de sorte que le lemme~\ref{lm-diffev} achève la démonstration.
\end{proof}

\subsubsection{Diagonalisabilité de l'algèbre d'un groupe}\label{par-diagrp}

On commence par examiner la question de la diagonalisabilité de l'algèbre d'un groupe cyclique fini :

\begin{pr}\label{pr-diag_Zn} Soit $n\in\mathbb{N}^*$. Les assertions suivantes sont équivalentes :
\begin{enumerate}
\item\label{itcy1} l'algèbre $k[\mathbb{Z}/n]$ du groupe cyclique $\mathbb{Z}/n$ est fortement diagonalisable ;
\item\label{itcy2} le polynôme $X^n-1$ de $k[X]$ se décompose en produit de polynômes unitaires de degré $1$ deux à deux étrangers ;
\item\label{itcy3} $n$ est inversible dans $k$ et le $n$-ième polynôme cyclotomique $\Phi_n$ possède une racine dans $k$.
\end{enumerate}
\end{pr}

\begin{proof} Comme l'algèbre $k[\mathbb{Z}/n]$ est isomorphe à $k[X]/(X^n-1)$, l'équivalence entre \ref{itcy1} et \ref{itcy2} constitue un cas particulier du lemme~\ref{quot-diag}. Ce même lemme montre que ces conditions impliquent que $X^n-1\in k[X]$ est étranger à son polynôme dérivé $nX^{n-1}$, ce qui implique que $n$ est inversible dans $k$ (évaluer en $1$). Le lemme~\ref{lm-dpolprod} montre par ailleurs que $k[X]/(\Phi_n(X))$ est une $k$-algèbre fortement diagonalisable si c'est le cas de $k[X]/(X^n-1)$. Ainsi, \ref{itcy2} entraîne \ref{itcy3}.

Supposons maintenant la condition \ref{itcy3} vérifiée ; soit $\xi$ une racine de $\Phi_n$. Pour tout $m\in\mathbb{N}$, $\xi^m$ annule le polynôme cyclotomique $\Phi_{n/\mathrm{pgcd}(n,m)}$. Si $m$ n'est pas divisible par $n$, on en déduit $\xi^m-1\in k^\times$ en appliquant le lemme~\ref{lm-inv_cycl}. Il s'ensuit que les polynômes $X-\xi^i$ ($0\le i<n$) sont deux à deux étrangers, grâce au lemme~\ref{lm-diffev}. Comme chacun d'entre eux divise $X^n-1$, leur produit divise également $X^n-1$ (lemme~\ref{lmev-g}), d'où $X^n-1=\prod_{i=0}^{n-1}(X-\xi^i)$. Cela établit l'implication \ref{itcy3}$\Rightarrow$\ref{itcy2} et termine la démonstration.
\end{proof}

\begin{rem}
Il existe un anneau pseudo-initial vérifiant les conditions de l'énoncé précédent, à savoir $\mathbb{Z}[1/n,X]/(\Phi_n(X))$ (qui se réalise concrètement comme le sous-anneau de $\mathbb{C}$ engendré par $\frac{1}{n}$ et $\exp(\frac{2\imath\pi}{n})$) : $k$ remplit les conditions si et seulement s'il existe un morphisme de l'anneau en question vers $k$.
\end{rem}

\begin{pr}\label{pr-diagalgrp} Soit $V$ un groupe. Les assertions suivantes sont équivalentes.
\begin{enumerate}
\item\label{itdg1} La $k$-algèbre $k[V]$ du groupe $V$ est fortement diagonalisable.
\item\label{itdg2} L'algèbre $k[V]$ est faiblement diagonalisable.
\item\label{itdg3} Le groupe $V$ est commutatif, fini, et son exposant (i.e. le ppcm de l'ordre de ses éléments) $n$ vérifie les conditions de la proposition~\ref{pr-diag_Zn}.
\end{enumerate}
\end{pr}

\begin{proof} Supposons que la condition \ref{itdg3} est vérifiée. Si $m$ est un diviseur de $n$, alors $X^m-1$ divise $X^n-1$, de sorte que l'algèbre $k[\mathbb{Z}/m]$ est fortement diagonalisable grâce au lemme~\ref{lm-dpolprod}. Comme $V$ est la somme directe d'un nombre fini de groupes cycliques dont les ordres divisent $n$, $k[V]$ est le produit tensoriel d'un nombre fini de $k$-algèbres fortement diagonalisables. On en déduit l'implication \ref{itdg3}$\Rightarrow$\ref{itdg1}  par le lemme~\ref{lm-ptd}.

Supposons maintenant que la condition \ref{itdg2} est vérifiée. Comme une algèbre faiblement diagonalisable est commutative et que $k$ est non nul, $V$ est nécessairement commutatif ; on le note additivement. Pour voir que $V$ est fini, on remarque que, dans une algèbre de la forme $k^E$ avec $E\ne\varnothing$, il existe un idempotent non nul $e$ tel que $ae$ soit proportionnel à $e$ pour tout $a$. Traduit dans $k[V]$, si l'on écrit $e=\sum_{u\in V}\lambda(u)[u]$, où les $\lambda(u)\in k$ sont presque tous nuls, cela fournit, pour tout $v\in V$, un scalaire $\varphi(v)$ tel que $[v]e=\varphi(v)e$, soit $\lambda(u-v)=\varphi(v)\lambda(u)$ pour tout $(u,v)\in V^2$. Comme il existe $t\in V$ tel que $\lambda(t)\ne 0$, prenant $v=u-t$, on voit que $\lambda(u)\ne 0$ pour tout $u\in V$, d'où la finitude de $V$.

Montrons maintenant que l'ordre $r$ de $V$ est inversible dans $k$ (ce qui est équivalent à demander que son exposant le soit) dans le cas particulier où $k$ est un corps. Les relations précédentes montrent que $\varphi$ est un morphisme de groupes $V\to k^\times$ et que $\lambda(v)=\lambda(0)\varphi(v)^{-1}$ pour tout $v\in V$. L'idempotence de $e$ s'écrit maintenant
$$e=e^2=\sum_{u\in V}\lambda(u)[u]e=\Big(\sum_{u\in V}\lambda(u)\varphi(u)\Big)e=r\lambda(0)e\;$$
ce qui implique que $r$ est inversible dans le corps $k$, puisque $e\ne 0$.

Montrons maintenant que $n$ est inversible dans $k$ dans le cas général : sinon, cet entier appartiendrait à un idéal maximal $\mathfrak{m}$ de $k$. En réduisant modulo $\mathfrak{m}$, on voit que la  $k/\mathfrak{m}$-algèbre $(k/\mathfrak{m})[V]$ est également faiblement diagonalisable, ce qui permet de se ramener au cas traité précédemment.

Il existe un isomorphisme $V\simeq\mathbb{Z}/n\oplus W$, où $n$ est l'exposant de $V$ et $W$ un groupe dont l'ordre est inversible dans $k$. Il s'ensuit que l'augmentation $k[W]\twoheadrightarrow k$ est un épimorphisme scindé (par l'idempotent $\frac{1}{|W|}\underset{g\in W}{\sum}[g]$) de $k[W]$-algèbres. En tensorisant avec $k[\mathbb{Z}/n]$, on en déduit un épimorphisme scindé de $k[V]$-modules $k[V]\twoheadrightarrow k[\mathbb{Z}/n]$. La proposition~\ref{pr-fcfd} montre alors que $k[\mathbb{Z}/n]$ est une $k$-algèbre fortement diagonalisable (comme $V$ est fini, l'hypothèse sur $k[V]$ implique que cette algèbre est \textit{fortement} diagonalisable). Cela achève d'établir l'implication \ref{itdg2}$\Rightarrow$\ref{itdg3}.

Comme l'implication \ref{itdg1}$\Rightarrow$\ref{itdg2} est triviale, cela termine la démonstration.
\end{proof}

\begin{rem} Si l'on suppose que $k$ est un corps (cas le plus classique), on peut raccourcir la démonstration, par exemple en utilisant la théorie élémentaire des algèbres semi-simples.
\end{rem}

\subsection{Inversion de Fourier pour les groupes abéliens finis}\label{par-Fou}

Il s'agit de rendre explicite, sous les hypothèses de la proposition~\ref{pr-diagalgrp}, l'isomorphisme de $k$-algèbres $k[V]\simeq k^{|V|}$.

Si $V$ est un groupe abélien fini, on note
$$V^\sharp:=\Hom_\mathbb{Z}(V,\mathbb{Q}/\mathbb{Z})$$
son dual (qui est isomorphe, non naturellement, à $V$). Ainsi, $(-)^\sharp$ est une équivalence de catégories entre la catégorie des groupes abéliens finis et sa catégorie opposée ; cette équivalence est son propre quasi-inverse. Si $v$ est un élément de $V$ et $l$ un élément de $V^\sharp$, on notera souvent $<v,l>$ pour $l(v)$ afin de faciliter la lecture. Les éléments de $V^\sharp$ sont à valeurs dans le sous-groupe cyclique de $\mathbb{Q}/\mathbb{Z}$ dont l'ordre est l'exposant de $V$.

\begin{nota} Soient $V$ un groupe abélien fini d'exposant $N$ et $T$ un sous-groupe de $\mathbb{Q}/\mathbb{Z}$ contenant son sous-groupe d'ordre $N$.
Si $\varepsilon : T\to k^\times$ est un morphisme de groupes, on note $\Phi_{V,\varepsilon} : k[V]\to k^{V^\sharp}$ le morphisme de $k$-algèbres donné par $\Phi_{V,\varepsilon}([v])(l):=\varepsilon(<v,l>)$ pour tout $(v,l)\in V\times V^\sharp$.
\end{nota}

\begin{defi} Soit $n\in\mathbb{N}^*$. Une \textit{$n$-racine primitive}  de l'unité dans $k$ est  une racine du $n$-ième polynôme cyclotomique de $k[X]$.
\end{defi}

\begin{rem} 
\begin{enumerate}
\item Une $n$-racine primitive de l'unité dans $k$ est un élément d'ordre divisant $n$ du groupe $k^\times$, et d'ordre exactement $n$ si $k$ est de caractéristique nulle ou étrangère à $n$.
\item Toute puissance d'une racine primitive de l'unité est une racine primitive de l'unité. Toutefois, contrairement au cas où $k$ est un corps, toute racine de l'unité n'est pas nécessairement primitive, même si son ordre est inversible dans $k$ (si $k$ n'est pas de caractéristique $2$, dans l'anneau produit $k^2$, l'élément $(1,-1)$ est une racine de l'unité d'ordre $2$ qui n'est pas primitive).
\end{enumerate}
\end{rem}

\begin{pr}[Inversion de Fourier pour les groupes abéliens finis]\label{pr-invF} Soit $V$ un groupe abélien fini d'exposant $e$.  Supposons que $e$ est inversible dans $k$, et que $\varepsilon : T\hookrightarrow k^\times$, où $T$ est un sous-groupe de $\mathbb{Q}/\mathbb{Z}$ contenant $\mathbb{Z}/e$, est un morphisme de groupes envoyant la classe de $\frac{1}{e}$ sur une $e$-racine primitive de l'unité dans $k$. Alors $\Phi_{V,\varepsilon} : k[V]\to k^{V^\sharp}$ est un isomorphisme de $k$-algèbres, dont l'inverse est donné par
$$\Psi_{V,\varepsilon} : f\in k^{V^\sharp}\mapsto\sum_{v\in V}\hat{f}(v)[v]\;,$$
où $\hat{f}$ est la \textbf{transformée de Fourier} de $f$ :
$$\hat{f} : V\to k\qquad v\mapsto \;\frac{1}{|V|}\sum_{l\in V^\sharp}f(l)\varepsilon(-<v,l>).$$
\end{pr}

\begin{proof} Soit $x\in V$ ; calculons la transformée de Fourier de $\Phi_{V,\varepsilon}([x])$ : évaluée en $v\in V$, elle vaut
$$\frac{1}{|V|}\sum_{l\in V^\sharp}\varepsilon(l(x)).\varepsilon(-l(v))=\frac{1}{|V|}\sum_{l\in V^\sharp}\varepsilon(l(x-v))\;,$$
qui vaut manifestement $1$ lorsque $v=x$. Si $v\ne x$, alors il existe $L\in V^\sharp$ tel que $L(x-v)\ne 1$ ; comme $l\mapsto l+L$ est une permutation de $V^\sharp$, on a $\widehat{\Phi_{V,\varepsilon}([x])}(v)=L(x-v).\widehat{\Phi_{V,\varepsilon}([x])}(v)$. Comme $L(x-v)$ est une racine de l'unité dont l'ordre divise $e$, et est donc inversible dans $k$, et est distincte de $1$, on a $L(x-v)-1\in k^\times$ par le lemme~\ref{lm-inv_cycl}. Il s'ensuit que $\widehat{\Phi_{V,\varepsilon}([x])}(v)=0$. Par conséquent, la composée $\Psi_{V,\varepsilon} \circ\Phi_{V,\varepsilon}$ est l'identité de $k[V]$.

La démonstration que $\Phi_{V,\varepsilon} \circ\Psi_{V,\varepsilon}$ est l'identité de $k^{V^\sharp}$ est entièrement analogue, ou se déduit de ce qui précède en remplaçant $V$ par $V^\sharp$, d'où le résultat.
\end{proof}

\subsection{Sommes de Gau\ss}\label{pgau}

\begin{nota} On note $\zcyc$ le sous-anneau de $\mathbb{C}$ engendré par les racines de l'unité.
\end{nota}

\begin{hyp} Dans tout le §\,\ref{pgau}, on suppose que $k$ est un sous-anneau de $\mathbb{C}$ contenant $\zcyc$, et $N$ désigne un entier strictement positif.
\end{hyp}

Si $\chi : (\mathbb{Z}/N)^\times\to k^\times$ et $\tau : \mathbb{Z}/N\to k^\times$ sont des morphismes de groupes, on note $\G_N(\chi,\tau)$ la \textit{somme de Gau\ss}
\begin{equation}\label{eq-dfgs}
\G_N(\chi,\tau):=\sum_{t\in(\mathbb{Z}/N)^\times}\chi(t)\tau(t)\;.
\end{equation}

Un morphisme $\chi : (\mathbb{Z}/N)^\times\to k^\times$ est aussi appelé \textit{caractère} modulo $N$. On dit que $\chi$ est \textit{imprimitif} s'il existe un diviseur strict $d$ de $N$ tel que $\chi$ se factorise à travers le morphisme canonique $(\mathbb{Z}/N)^\times\to(\mathbb{Z}/d)^\times$ ; dans le cas contraire, $\chi$ est dit \textit{primitif}.

\begin{nota}\label{not-eps} Pour $u\in\mathbb{N}$, on note $\varepsilon_u : \mathbb{Z}/N\to k^\times$ le morphisme $\bar{t}\mapsto\exp(\frac{2\imath\pi tu}{N})$, où $\bar{t}$ désigne la classe modulo $N$ de $t\in\mathbb{Z}$. On note simplement $\varepsilon$ pour $\varepsilon_1$.
\end{nota}

Pour la démonstration du résultat classique suivant, voir par exemple \cite[th.~9.7]{MV}.
\begin{pr}\label{pr-gprim} Supposons $N$ inversible dans $k$. Soient $\chi : (\mathbb{Z}/N)^\times\to k^\times$ un caractère primitif et $u\in\mathbb{Z}$.
\begin{enumerate}
\item Si $u$ et $N$ sont étrangers, alors $\G_N(\chi,\varepsilon_u)\in k^\times$ ; de plus, $\G_N(\chi,\varepsilon_u)=\chi(u)^{-1}\G_N(\chi,\varepsilon)$.
\item Dans le cas contraire, $\G_N(\chi,\varepsilon_u)=0$.
\end{enumerate}
\end{pr}

Nous aurons également besoin du résultat suivant sur les caractères imprimitifs, qui constitue un cas particulier de \cite[th.~9.10]{MV}.

\begin{pr}\label{pr-gimp} Supposons que $N=p^r$ avec $r>1$, que $\chi : (\mathbb{Z}/N)^\times\to k^\times$ est un caractère imprimitif et que $\tau : \mathbb{Z}/N\to k^\times$ est un morphisme injectif. Alors $\G_N(\chi,\tau)=0$.
\end{pr}

\begin{pr}\label{pr-gimp2} Supposons que $N=p^r$ avec $r>1$, que $\chi : (\mathbb{Z}/N)^\times\to k^\times$ est un caractère imprimitif et $\tau : \mathbb{Z}/N\to k^\times$ est un morphisme non injectif. Notons $\bar{\chi} :  (\mathbb{Z}/p^{r-1})^\times\to k^\times$ (resp. $\bar{\tau} : \mathbb{Z}/p^{r-1}\to k^\times$) le morphisme de groupes induit par $\chi$ (resp. $\tau$). Alors $\G_{p^r}(\chi,\tau)=p.\G_{p^{r-1}}(\bar{\chi},\bar{\tau})$.
\end{pr}

\begin{proof} Cela résulte de la définition et du fait qu'un élément de $\mathbb{Z}/p^r$ est inversible si et seulement si son image dans $\mathbb{Z}/p^{r-1}$ est inversible.
\end{proof}

\section{Le problème de l'isomorphisme fonctoriel $k[A]\simeq k^B$ pour des foncteurs additifs $A$ et $B$}

\begin{nota} \textbf{Dans toute la suite de cet article, $\A$ désigne une catégorie additive essentiellement petite.}
\end{nota}

%Nous rencontrerons particulièrement les catégories additives suivantes à la source :

\begin{nota} On désigne par $\mathbf{Ab}$ la catégorie des groupes abéliens et par $\mathbf{Ab}^f$ sa sous-catégorie pleine constituée des groupes abéliens finis. Pour tout entier $N>0$, on note $\mathbf{Ab}^f_N$ la sous-catégorie pleine de $\mathbf{Ab}^f$ des groupes abéliens finis annulés par $N$. On note $\mathbf{Ab}^f_{(p)}$ la sous-catégorie pleine de $\mathbf{Ab}^f$ des $p$-groupes abéliens finis (où $p$ est un nombre premier).
\end{nota}

\begin{nota}
Si $\C$ est une catégorie essentiellement petite, on désigne par $\F(\C;k)$ la catégorie dont les objets sont les foncteurs de $\C$ vers $k\Md$ et les morphismes les transformations naturelles.
\end{nota}

Si $A : \A\to\mathbf{Ab}$ est un foncteur additif, on note $k[A]$ le foncteur de $\F(\A;k)$ donné par la composée $\A\xrightarrow{A}\mathbf{Ab}\to\mathbf{Ens}\xrightarrow{k[-]}k\Md$, où la flèche centrale est le foncteur d'oubli, et $k^A$ le foncteur de $\F(\A^\op;k)$ donné par la composée $\A^\op\xrightarrow{A}\mathbf{Ab}^\op\to\mathbf{Ens}^\op\xrightarrow{k^{(-)}}k\Md$. Ainsi $k^A\simeq\Hom_k(-,k)\circ k[A]$.

\subsection{Additivisation}

Les lemmes élémentaires qui suivent sont bien connus et peuvent s'étendre aux approximations polynomiales de degré supérieur des foncteurs (dont nous n'aurons jamais usage ici) --- cf. par exemple \cite[§\,4.2.1]{DT-poids}. 

\begin{lm} Soit $F$ un foncteur de $\F(\A;k)$. Le plus grand quotient additif de $F$ est le foncteur $F_\mathrm{add}$ donné par 
$$F_\mathrm{add}(a):=\mathrm{Coker}\big(F(a\oplus a)\xrightarrow{F(p_1)+F(p_2)-F(\nabla)} F(a)\big)\;$$
où $p_i : a\oplus a\to a$ ($i\in\{1,2\}$) désigne la $i$-ème projection et $\nabla : a\oplus a\to a$ la somme.
\end{lm}

\begin{proof}
C'est une conséquence formelle directe de ce que $F$ est additif si et seulement si $F(p_1)+F(p_2)-F(\nabla)=0$.
\end{proof}

l'énoncé suivant est analogue au (et dual du) précédent.

\begin{lm} Soit $F$ un foncteur de $\F(\A;k)$. Le plus grand sous-foncteur additif de $F$ est le foncteur $F^\mathrm{add}$ donné par 
$$F^\mathrm{add}(a):=\mathrm{Ker}\big(F(a)\xrightarrow{F(i_1)+F(i_2)-F(\Delta)} F(a\oplus a)\big)\;$$
où $i_j : a\to a\oplus a$ ($j\in\{1,2\}$) désigne la $i$-ème inclusion et $\Delta : a\to a\oplus a$ la diagonale.
\end{lm}

Les deux énoncés suivants s'obtiennent à partir des précédents par un calcul direct.

\begin{lm}\label{lmadd1} Soit $A : \A\to\mathbf{Ab}$ un foncteur additif. On a $k[A]_\mathrm{add}\simeq k\otimes_\mathbb{Z} A$ et $k[A]^\mathrm{add}=0$.
\end{lm}

\begin{lm}\label{lmadd2}  Soit $B : \A^\op\to\mathbf{Ab}$ un foncteur additif. On a $(k^B)_\mathrm{add}=0$ et $(k^B)^\mathrm{add}\simeq\Hom_\mathbb{Z}(B,k)$.
\end{lm}

Il résulte des lemmes~\ref{lmadd1} et \ref{lmadd2} que :

\begin{pr}\label{pr-ckt} Soient $\A$ une petite catégorie additive, $A : \A\to\mathbf{Ab}$ et $B : \A^\op\to\mathbf{Ab}$ des foncteurs additifs. Si les foncteurs $k[A]$ et $k^B$ de $\F(\A;k)$ sont isomorphes, alors les foncteurs additifs $k\otimes_\mathbb{Z} A$ et $\Hom_\mathbb{Z}(B,k)$ sont nuls.
\end{pr}

Cela nous conduit à introduire la notion suivante, directement inspirée de \cite[déf.~4.1]{DTV}, dont nous reprenons la terminologie.

\begin{defi} Un foncteur additif $A : \A\to\mathbf{Ab}$ est dit \textbf{$k$-trivial} s'il est à valeurs dans les groupes abéliens finis et que le foncteur $k\otimes_\mathbb{Z} A$ est nul.
\end{defi}

L'ajout de la condition des valeurs finies, par rapport à la proposition~\ref{pr-ckt}, provient du fait qu'il est beaucoup plus facile de traiter de la linéarisation de foncteurs additifs lorsqu'ils sont à valeurs finies que dans le cas général. Il importe que noter que, si le fait que la $k$-algèbre d'un groupe soit isomorphe à un produit (a priori infini) de copies de $k$ implique que ce groupe est nécessairement fini (cf. proposition~\ref{pr-diagalgrp}), l'existence d'un isomorphisme $k[A]\simeq k^B$ pour des foncteurs additifs $A$ et $B$ n'entraîne \emph{pas} que $A$ ou $B$ soit à valeurs finies, comme nous le verrons au §\,\ref{ssect-vi}. Ce n'est que dans le cas de foncteurs à valeurs finies que nous pourrons caractériser complètement l'existence d'un tel isomorphisme (en combinant la proposition~\ref{pr-ckt} et le théorème~\ref{thprinc} ci-après).

\subsection{Morphismes}

L'énoncé élémentaire et classique suivant figure verbatim dans \cite[lemme~C.1]{DTV}.

\begin{pr} Soient $A, B : \A\to\mathbf{Ab}$ des foncteurs additifs. Le morphisme $k$-linéaire canonique
$$k[\Hom(A,B)]\to\Hom(k[A],k[B])$$
induit par la post-composition par $k[-]$ est injectif ; il est bijectif si $A$ est de type fini.
\end{pr}

\begin{rem} On déduit facilement de cette proposition que si $A, B : \A\to\mathbf{Ab}$ sont des foncteurs additifs à valeurs finies tels que $k[A]\simeq k[B]$, \textit{et que $A$ (ou $B$) est de type fini}, alors $A\simeq B$. Nous ignorons s'il est possible de s'affranchir de l'hypothèse que l'un des foncteurs $A, B$ est de type fini (et de l'hypothèse de valeurs finies).
\end{rem}

\begin{nota} Si $R$ est un anneau, on désigne par $R_\mu$ le monoïde multiplicatif sous-jacent à $R$.
\end{nota}

\begin{cor} Soit $A : \A\to\mathbf{Ab}$ un foncteur additif. Le morphisme de $k$-algèbres canonique
$$k[\End(A)_\mu]\to\End(k[A])$$
est injectif ; il est bijectif si $A$ est de type fini.
\end{cor}

Le cas particulier suivant (qu'on peut établir encore plus directement par une simple invocation du lemme de Yoneda) nous sera utile par la suite.

\begin{cor}\label{cor-end_eval} Soit $N\in\mathbb{N}^*$. L'application linéaire
$$k[(\mathbb{Z}/N)_\mu]\to\End_{\F(\mathbf{Ab}^f_N;k)}(k[A])\qquad [a]\mapsto\big(k[V]\to k[V]\quad [v]\mapsto [av]\big)$$
est un isomorphisme de $k$-algèbres, où $A : \mathbf{Ab}^f_N\to\mathbf{Ab}$ est le foncteur d'oubli, dont l'inverse est induit par l'évaluation en $\mathbb{Z}/N$.
\end{cor}

On rappelle que si $\C$ est une catégorie essentiellement petite, $T$ un foncteur de $\F(\C^\op;k)$ et $F$ un foncteur de $\F(\C;k)$, on définit le produit tensoriel de $T$ et $F$ au-dessus de $\C$ comme le $k$-module $T\underset{\C}{\otimes}F$ coend du bifoncteur produit tensoriel extérieur $T\boxtimes F : \C^\op\times\C\to k\Md\quad (x,y)\mapsto T(x)\otimes_k F(y)$ (cf. par exemple \cite[§\,2.2 et seq.]{DT-excis}) ; ce module est caractérisé par l'isomorphisme $k$-linéaire naturel
\begin{equation}\label{eq-ptcat}
\Hom_k(T\underset{\C}{\otimes}F,M)\simeq\Hom_{\F(\C;k)}(F,\Hom_k(-,M)\circ T)
\end{equation}
(où $M$ est un $k$-module).
Si $x$ est un objet de $\C$ et $t$ un élément du $k$-module $T(x)\otimes_k F(x)$, on note $\llbracket t\rrbracket$ l'image canonique de $t$ dans $T\underset{\C}{\otimes}F$.

\begin{pr}\label{pr-ptend} Soient $A : \A\to\mathbf{Ab}$ et $B : \A^\op\to\mathbf{Ab}$ des foncteurs additifs. Le morphisme
$$k^{B\underset{\A}{\otimes} A}\to\Hom(k[A],k^B)$$
associant à $\alpha : B\underset{\A}{\otimes} A\to k$ la transformation naturelle $k[A]\to B$ qui, évaluée sur un objet $x$ de $\A$, envoie $[u]$ (pour $u\in A(x)$) sur $(v\in B(x))\mapsto\alpha(\llbracket v\otimes u\rrbracket)$,
est bijectif.
\end{pr}

\begin{proof} Cet isomorphisme se déduit directement de \cite[lemma~3.5]{DT-excis} et de l'isomorphisme de dualité \eqref{eq-ptcat}.
\end{proof}

%Après : se concentrer sur la situation universelle des groupes abéliens $k$-triviaux, bien donner de jolies notations pour les flèches canoniques, et faire la décomposition primaire rapidement qui permet de se ramener à des $p$-groupes abéliens finis (avec $p\in k^\times$).

Le cas particulier suivant de la proposition~\ref{pr-ptend} sert avant tout à fixer des notations que nous utiliserons ultérieurement.

\begin{cor}\label{cor-Phitr} Soit $\A$ l'une des catégories $\mathbf{Ab}^f$, $\mathbf{Ab}^f_N$ (où $N\in\mathbb{N}^*$) ou $\mathbf{Ab}^f_{(p)}$. On note $A : \A\to\mathbf{Ab}$ le foncteur d'inclusion et $T$ le groupe $\mathbb{Q}/\mathbb{Z}$ (pour $\A=\mathbf{Ab}^f$), $\mathbb{Z}/N$ (pour $\A=\mathbf{Ab}^f_N$) ou $\mathbb{Z}/p^\infty$ (pour $\A=\mathbf{Ab}^f_{(p)}$). Alors l'application $k$-linéaire
$$\Phi : k^T\to\Hom_{\F(\A;k)}(k[A],k^{A^\sharp})$$
donnée par
$$\Phi(\alpha)_V([v]):=(l\in V^\sharp\mapsto\alpha(<v,l>))$$
pour tous $V\in\A$, $v\in V$ et toute fonction $\alpha : T\to k$, est un isomorphisme.
\end{cor}

\begin{proof} Contentons-nous de traiter le cas $\A=\mathbf{Ab}^f$, les autres étant analogues (et plus simple pour $\A=\mathbf{Ab}^f_N$). Le foncteur $A$ est alors la colimite sur $n\in\mathbb{N}^*$ (pour l'ordre de divisibilité) des foncteurs $\Hom(\mathbb{Z}/n,-)$, d'où $A^\sharp\underset{\A}{\otimes} A\simeq\underset{n\in\mathbb{N}^*}{\col}A^\sharp(\mathbb{Z}/n)\simeq\underset{n\in\mathbb{N}^*}{\col}\mathbb{Z}/n\simeq\mathbb{Q}/\mathbb{Z}$. La proposition~\ref{pr-ptend} permet de conclure.
\end{proof}

La proposition~\ref{pr-invF} montre que si le foncteur d'inclusion $A$ est $k$-trivial (i.e. que $k$ contient $\mathbb{Q}$ si $\A=\mathbf{Ab}^f$, $N$ est inversible dans $k$ si $\A=\mathbf{Ab}^f_N$, ou $p$ est inversible dans $k$ si $\A=\mathbf{Ab}^f_{(p)}$) et que $\varepsilon : T\to k^\times$ est un monomorphisme de groupes dont l'image est constituée de racines primitives de l'unité, alors $\Phi(\varepsilon)$ (où l'on note par abus encore $\varepsilon$ la fonction ensembliste composée de ce morphisme et de l'inclusion $k^\times\to k$) est un isomorphisme.

\subsection{Résultat principal : énoncé ; premières réductions}

Le théorème suivant constitue le principal résultat de cet article (et le seul nouveau, avec certains énoncés du §\,\ref{ssect-vi}) ; il sera démontré un peu plus tard. Il généralise la proposition~3.3 de \cite{DG} (qui joue un rôle très important dans ledit article\,\footnote{On notera que, dans \cite{DG}, le fait que l'isomorphisme $k[A]\simeq k^{A^\sharp}$ pour $A$ foncteur additif $k$-trivial ne soit montré que sous des hypothèses supplémentaires sur $A$ ou sur $k$ n'a pas d'incidence sur la généralité des résultats ultérieurs, qui traitent le plus souvent de foncteurs $A$ de type fini, et pour lesquels, de plus, l'extension des scalaires au but ne pose pas de problème particulier.} --- et donc sur la structure des foncteurs \textit{antipolynomiaux}, voire plus généraux --- cf. \cite[§\,4.1 et 4.2]{DTV}) en supprimant l'hypothèse que $k$ est un corps et surtout celle que $A$ est un foncteur de type fini (ou celle que $k$ contient assez de racines de l'unité --- cf.  \cite[rem.~3.5]{DG}), ce qui complique assez nettement la démonstration, même si celle-ci demeure élémentaire.

\begin{thm}\label{thprinc} Soit $A : \A\to\mathbf{Ab}$ un foncteur additif $k$-trivial. Alors il existe dans $\F(\A;k)$ un isomorphisme $k[A]\simeq k^{A^\sharp}$.
\end{thm}

Avant de donner quelques réductions simples préalables à la démonstration de ce théorème, introduisons une notation, dans laquelle L abrège \textit{linéarisation} et D \textit{dual}.

\begin{nota} On dit qu'un foncteur additif $A : \A\to\mathbf{Ab}$ vérifie la propriété $(\mathrm{LD})_k$ si $k[A]\simeq k^{A^\sharp}$ dans $\F(\A;k)$.
\end{nota}

\begin{lm}\label{lm-sdi} Soit $A_i : \A\to\mathbf{Ab}$ ($i\in E$) une famille de foncteurs additifs. On suppose que, pour tout objet $t$ de $\A$, le sous-ensemble $\{i\in E\,|\,A_i(t)\ne 0\}$ de $E$ est fini. 

Si chaque $A_i$ vérifie la propriété $(\mathrm{LD})_k$, alors $\bigoplus_{i\in E} A_i$ la vérifie également.
\end{lm}

\begin{proof} L'hypothèse de finitude faite garantit que
\begin{equation}\label{eq-dsdi}\Big(\bigoplus_{i\in E} A_i\Big)^\sharp\simeq\bigoplus_{i\in E} A_i^\sharp\;.
\end{equation}
On a par ailleurs
\begin{equation}\label{eq-ptil}k[\bigoplus_{i\in E} A_i]\simeq\bigotimes_{i\in E} k[A_i]\,,
\end{equation}
où le produit tensoriel signifie la colimite (filtrante) sur les parties finies $E'$ de $E$ des $\bigotimes_{i\in E'} k[A_i]$, les applications de transition, pour $E'\subset E''$, étant obtenues par produit tensoriel des identités de $k[A_i]$ pour $i\in E'$ et de l'unité $k\to k[A_i]$ pour $i\in E''\setminus E'$.

On a de même, en utilisant \eqref{eq-dsdi},
\begin{equation}\label{eq-ptdl}
k^{\big(\bigoplus_{i\in E} A_i\big)^\sharp}\simeq\bigotimes_{i\in E} k^{A_i^\sharp}
\end{equation}
ce qui permet de conclure, car si $A$ vérifie la condition $(\mathrm{LD})_k$, le scindement naturel des parties constantes des foncteurs montre qu'on peut trouver un isomorphisme $k[A]\simeq k^{A^\sharp}$ faisant commuter le diagramme
$$\xymatrix{k\ar[r]\ar[rd] & k[A]\ar[d]^\simeq \\
& k^{A^\sharp}
}$$
dont les flèches de source $k$ sont les unités. La conclusion découle donc de \eqref{eq-ptil} et \eqref{eq-ptdl}.
\end{proof}

\begin{lm}\label{lm-escal} Soient $k'\to k$ un morphisme d'anneaux commutatifs et $A : \A\to\mathbf{Ab}$ un foncteur additif à valeurs finies. Si $A$ vérifie la condition $(\mathrm{LD})_{k'}$, alors il vérifie $(\mathrm{LD})_k$.
\end{lm}

\begin{proof} Dans $\F(\A;k)$, on a $k'[A]\otimes_{k'}k\simeq k[A]$ et, comme $A$ est à valeurs finies, $(k')^{A^\sharp}\otimes_{k'}k\simeq k^{A^\sharp}$ Le lemme découle donc d'une extension des scalaires au but.
\end{proof}

\begin{lm}\label{lm-cauni} Si, pour tout nombre premier $p$, le foncteur d'inclusion $\mathbf{Ab}^f_{(p)}\to\mathbf{Ab}$ vérifie la propriété $(\mathrm{LD})_{\mathbb{Z}[1/p]}$, alors le théorème~\ref{thprinc} est vrai.
%
 %Si le théorème~\ref{thprinc} est vrai, pour tout nombre premier $p$, lorsque $k=\mathbb{Z}[1/p]$, que $\A$ est la catégorie des $p$-groupes abéliens finis et que $A$ est le foncteur d'inclusion, alors il est valide dans le cas général.
\end{lm}

\begin{proof} Par précomposition, l'hypothèse implique tout foncteur additif $\A\to\mathbf{Ab}$ à valeurs dans $\mathbf{Ab}^f_{(p)}$ vérifie $(\mathrm{LD})_{\mathbb{Z}[1/p]}$, et donc aussi $(\mathrm{LD})_k$ si $p$ est inversible dans $k$, par le lemme~\ref{lm-escal}.

 Le cas général d'un foncteur additif $k$-trivial $A : \A\to\mathbf{Ab}$ s'en déduit grâce à la décomposition primaire de $A$ et au lemme~\ref{lm-sdi}, en notant que la condition $k\otimes_\mathbb{Z} A=0$ implique que $p$ est inversible dans $k$ pour tout nombre premier $p$ tel que la composante $p$-primaire de $A$ soit non nulle.
\end{proof}

%\begin{lm} Supposons que, pour tout nombre premier $p$ inversible dans $k$, $\alpha_p : \mathbb{Z}/p^\infty\to k$ est une fonction telle que $\alpha_p(0)=1$. Soit $\alpha : \bigoplus_p \mathbb{Z}/p^\infty\to k$, où la somme directe est prise sur les nombres premiers $p$ inversibles dans $k$, la fonction donnée par $(v_p)\mapsto\prod_p\alpha_p(v)$ (tous les termes du produit sont égaux à $1$ sauf un nombre fini, celui-ci fait donc sens). Alors
%\end{lm}

L'énoncé suivant reprend les notations du corollaire~\ref{cor-Phitr}. 

\begin{lm}\label{lm-cab} Soit $A : \mathbf{Ab}^f_{(p)}\to\mathbf{Ab}$ le foncteur d'oubli. Soit $K$ une $k$-algèbre (unitaire, associative et) commutative. Supposons que $K$ est un $k$-module libre non nul. Soit $\alpha : T\to k$ une fonction ensembliste et $\tilde{\alpha}$ sa composée avec le morphisme canonique $k\to K$. Alors le morphisme $\Phi(\alpha) : k[A]\to k^{A^\sharp}$ de $\F(\mathbf{Ab}^f_{(p)};k)$ est un isomorphisme si et seulement si le morphisme $\Phi(\tilde{\alpha}) : K[A]\to K^{A^\sharp}$ de $\F(\mathbf{Ab}^f_{(p)};K)$ est un isomorphisme
\end{lm}

\begin{proof} Comme $A$ est à valeurs finies, $\Phi(\tilde{\alpha})$ s'identifie à $\Phi(\alpha)\otimes_k K$, la conclusion résulte donc de ce que $K$ est fidèlement plat sur $k$, puisque libre et non nul.
\end{proof}

\subsection{Résultat principal : démonstration}

\begin{lm}\label{lm-invloc} Soit $N$ une puissance d'un nombre premier $p$. Un élément de $k[(\mathbb{Z}/N)_\mu]$ est inversible si et seulement si son image dans $k\times k[(\mathbb{Z}/N)^\times]$ par le morphisme d'anneaux dont les composantes sont l'augmentation $k[(\mathbb{Z}/N)_\mu]\to k$ et le morphisme $k[(\mathbb{Z}/N)_\mu]\to k[(\mathbb{Z}/N)^\times]$ donné par $[t]\mapsto [t]$ pour $t\in(\mathbb{Z}/N)^\times$ et $[t]\mapsto 0$ pour $t\notin(\mathbb{Z}/N)^\times$ est inversible.
\end{lm}

\begin{proof} En effet, le noyau de ce morphisme est l'idéal engendré par $[p]-[0]$, qui est nilpotent.
\end{proof}

\begin{lm}\label{lm-invag} Supposons que $k$ est un sous-anneau de $\mathbb{C}$ contenant $\zcyc$. Soit $G$ un groupe abélien fini. Un élément $\sum_{g\in G}c_g[g]$ de l'algèbre $k[G]$ est inversible si et seulement si, pour tout morphisme de groupes $\chi : G\to k^\times$, l'élément $\sum_{g\in G}c_g\chi(g)$ de $k$ est inversible.
\end{lm}

\begin{proof} Notons $\eta : k[G]\to k$ l'augmentation et, pour $\chi\in\mathbf{Ab}(G,k^\times)$, $\tilde{\chi}$ l'automorphisme de la $k$-algèbre $k[G]$ envoyant $[g]$ sur $\chi(g)[g]$ pour tout $g\in G$. Le morphisme d'algèbres $\eta\circ\tilde{\chi} : k[G]\to k$ est donné par  $\sum_{g\in G}c_g[g]\mapsto\sum_{g\in G}c_g\chi(g)$, ce qui montre que $\sum_{g\in G}c_g\chi(g)$ est nécessairement inversible dans $k$ si $\sum_{g\in G}c_g[g]$ est inversible dans $k[G]$.

Comme $k$ contient $\zcyc$, $\mathbf{Ab}(G,k^\times)\simeq G^\sharp$, et $k$ est exactement la sous-algèbre de $k[G]$ des éléments invariants par tous les $\tilde{\chi}$. Comme de plus $\widetilde{\chi_1}\circ\widetilde{\chi_2}=\widetilde{\chi_1\chi_2}$, il s'ensuit que, pour tout $x\in k[G]$,
$$\prod_{\chi\in\mathbf{Ab}(G,k^\times)}\tilde{\chi}(x)\in k,\text{ donc }\prod_{\chi\in\mathbf{Ab}(G,k^\times)}\tilde{\chi}(x)=\prod_{\chi\in\mathbf{Ab}(G,k^\times)}(\eta\circ\tilde{\chi})(x)\;.$$
Ainsi, si tous les $(\eta\circ\tilde{\chi})(x)$ sont inversibles dans $k$, tous les $\tilde{\chi}(x)$, et en particulier $x$, sont inversibles dans $k[G]$.
\end{proof}

\begin{rem} Lorsque $|G|$ est inversible dans $k$, le lemme résulte directement de l'inversion de Fourier.
\end{rem}

L'énoncé suivant reprend les notations du corollaire~\ref{cor-Phitr}, ainsi que celles du §\,\ref{par-Fou} ; la transformation de Fourier est relative au monomorphisme de groupes canonique $\varepsilon : \mathbb{Z}/N\to\zcyc[1/N]\quad\bar{t}\mapsto\exp(\frac{2\imath\pi t}{N})$ (cf. notation~\ref{not-eps}). On identifie de plus le groupe $\mathbb{Z}/N$ à son dual de la manière usuelle, en associant à $x\in\mathbb{Z}/N$ la multiplication par $x$ ; ainsi, la transformée de Fourier d'une fonction $\mathbb{Z}/N\to\zcyc[1/N]$ est une fonction $\mathbb{Z}/N\to\zcyc[1/N]$.

\begin{pr}\label{pr-inv1} Soit $N$ une puissance d'un nombre premier ; supposons que $k=\zcyc[1/N]$. Soit $\alpha : \mathbb{Z}/N\to k$ une fonction ensembliste. Alors $\Phi(\alpha) : k[A]\to k^{A^\sharp}$ est un isomorphisme si et seulement si $\alpha(0)\in k^\times$ et que, pour tout morphisme de groupes $\chi : (\mathbb{Z}/N)^\times\to k^\times$, on a
$$\sum_{t\in (\mathbb{Z}/N)^\times}\chi(t)\hat{\alpha}(t)\in k^\times\;.$$
\end{pr}

\begin{proof} Comme $\Phi(\varepsilon) : k[A]\to k^{A^\sharp}$ est un isomorphisme (proposition~\ref{pr-invF}), il s'agit donc de voir quand l'endomorphisme $\Phi(\varepsilon)^{-1}\circ\Phi(\alpha)$ de $k[A]$ est automorphisme, ou encore, par le corollaire~\ref{cor-end_eval}, quand son évaluation sur $\mathbb{Z}/N$ est un automorphisme $k$-linéaire de $k[\mathbb{Z}/N]$, i.e. envoie $[\bar{1}]$ sur un élément inversible de l'algèbre $k[(\mathbb{Z}/N)_\mu]$.

De manière générale, si $V$ est un objet de $\mathbf{Ab}^f_N$ et $v$ un élément de $V$,  $\Phi(\varepsilon)^{-1}\circ\Phi(\alpha)$ est donné sur $V$ par
\begin{equation}\label{eq-fct_fond}
[v]\mapsto\frac{1}{|V|}\underset{l\in V^\sharp}{\sum_{u\in V}}\alpha(l(v))\varepsilon(-l(u)) [u]=\frac{1}{|V|}{\sum_{u\in V}}\Big(\sum_{l\in V^\sharp}\alpha(l(v))\varepsilon(-l(u))\Big)[u]\;.
\end{equation}
Lorsque $V=\mathbb{Z}/N$ et $v=\bar{1}$, on a
$$\sum_{l\in V^\sharp}\alpha(l(v))\varepsilon(-l(u))=\sum_{x\in\mathbb{Z}/N}\alpha(x)\varepsilon(-xu)=N\hat{\alpha}(u)\;;$$
ainsi, $\Phi(\alpha)$ est un isomorphisme si et seulement si l'élément ${\sum_{u\in\mathbb{Z}/N}}\hat{\alpha}(u)[u]$ de l'algèbre $k[(\mathbb{Z}/N)_\mu]$ est inversible. Par le lemme~\ref{lm-invloc}, cela équivaut à demander que l'élément ${\sum_{u\in (\mathbb{Z}/N)^\times}}\hat{\alpha}(u)[u]$ de l'algèbre $k[(\mathbb{Z}/N)à^\times]$ soit inversible, ainsi que l'élément ${\sum_{u\in\mathbb{Z}/N}}\hat{\alpha}(u)=\alpha(0)$ (cf. proposition~\ref{pr-invF}) de $k$. Le lemme~\ref{lm-invag} permet de conclure.
\end{proof}

Pour étudier la condition d'inversibilité qui apparaît dans la proposition~\ref{pr-inv1}, on note que la définition de la transformée de Fourier et une interversion de sommations permettent de faire apparaître une somme de Gau\ss\ (cf. \eqref{eq-dfgs}) :
\begin{equation}\label{eq-avecG}
\sum_{t\in (\mathbb{Z}/N)^\times}\chi(t)\hat{\alpha}(t)=\frac{1}{N}\sum_{u\in\mathbb{Z}/N}\G_N(\chi,\varepsilon_u)\alpha(u)\;.
\end{equation}

La suite de l'analyse distingue selon que le caractère $\chi$ est primitif ou imprimitif.

\begin{lm}\label{lm-prim} Soit $N\in\mathbb{N}^*$ ; supposons que $k=\zcyc[1/N]$. Soient $\chi : (\mathbb{Z}/N)^\times\to k^\times$ un caractère primitif et $\alpha : \mathbb{Z}/N\to k$ une fonction ensembliste. Alors
$$\left(\sum_{t\in (\mathbb{Z}/N)^\times}\chi(t)\hat{\alpha}(t)\in k^\times\right)\Leftrightarrow\left(\sum_{u\in (\mathbb{Z}/N)^\times}\chi(u)^{-1}\alpha(u)\in k^\times\right)\;.$$
\end{lm}

\begin{proof} Cela résulte de la formule \eqref{eq-avecG} et de la proposition~\ref{pr-gprim}.
\end{proof}

\begin{lm}\label{lm-imprim} Soit $N=p^r$ avec $r\in\mathbb{N}^*$ ; supposons que $k=\zcyc[1/N]$. Soient $\chi : (\mathbb{Z}/N)^\times\to k^\times$ un caractère imprimitif et $\alpha : \mathbb{Z}/N\to k$ une fonction ensembliste. 
\begin{enumerate}
\item Si $r=1$, alors
$$\left(\sum_{t\in (\mathbb{Z}/N)^\times}\chi(t)\hat{\alpha}(t)\in k^\times\right)\Leftrightarrow\left(\sum_{u\in (\mathbb{Z}/p)^\times}\big(\alpha(u)-\alpha(0)\big)\in k^\times\right)\;.$$
\item Si $r>1$, notons $\bar{\alpha}$ la fonction composée du monomorphisme $\mathbb{Z}/p^{r-1}\hookrightarrow\mathbb{Z}/p^r$ et de $\alpha$, et $\bar{\chi} : (\mathbb{Z}/p^{r-1})^\times\to k^\times$ le caractère induit par $\chi$. Alors
$$\left(\sum_{t\in (\mathbb{Z}/N)^\times}\chi(t)\hat{\alpha}(t)\in k^\times\right)\Leftrightarrow\left(\sum_{t\in (\mathbb{Z}/p^{r-1})^\times}\bar{\chi}(t)\hat{\bar{\alpha}}(t)\in k^\times\right)\;.$$
\end{enumerate}
\end{lm}

\begin{proof} Si $N=p$ est premier, le seul caractère imprimitif $\chi$ modulo $N$ est le caractère trivial. On a alors $\G_p(\chi,\varepsilon_u)=\sum_{t\in (\mathbb{Z}/p)^\times}\varepsilon_u(t)$, qui vaut $-1$ si $u$ et $p$ sont étrangers (car $\sum_{t\in\mathbb{Z}/p}\varepsilon_u(t)=0$), et $p-1$ sinon. La première assertion résulte donc de \eqref{eq-avecG}.

La deuxième assertion découle quant à elle de  \eqref{eq-avecG} ainsi que des propositions~\ref{pr-gimp} et~\ref{pr-gimp2}.
\end{proof}

En combinant la proposition~\ref{pr-inv1} et les lemmes~\ref{lm-prim} et~\ref{lm-imprim}, et en tenant compte de l'égalité $\sum_{t\in (\mathbb{Z}/N)^\times}\chi(t)=0$ lorsque $\chi$ est un caractère modulo $N$ non trivial, on obtient :

\begin{pr} Supposons que $k=\zcyc[1/p]$. Soient $i\in\mathbb{N}^*\cup\{\infty\}$ et $\alpha : \mathbb{Z}/p^i\to k$ une fonction ensembliste. Alors $\Phi(\alpha)_V : k[V]\to k^{V^\sharp}$ est un isomorphisme pour tout groupe abélien fini $V$ annulé par $p^i$ si et seulement si les trois conditions suivantes sont satisfaites :
\begin{enumerate}
\item $\alpha(0)\in k^\times$ ;
\item $\sum_{i=1}^{p-1}\big(\alpha(i/p)-\alpha(0)\big)\in k^\times$ ;
\item pour tout entier $1\le r\le i$ et tout caractère primitif $\chi : (\mathbb{Z}/p^r)^\times\to k^\times$,
$$\sum_{t\in(\mathbb{Z}/p^r)^\times}\chi(t)^{-1}(\alpha(t)-\alpha(0))\in k^\times\;.$$
\end{enumerate}
\end{pr}

(Les notations $\Phi$ et $A$ apparaissant dans les énoncés précédent et suivant sont celles du corollaire~\ref{cor-Phitr}.)

\begin{cor}\label{cor-tpzc} Supposons que $k=\zcyc[1/p]$. Soit $\alpha : \mathbb{Z}/p^\infty\to k$ la fonction envoyant $1/p^r$ sur $2$ pour tout $r\in\mathbb{N}^*$ et tous les autres éléments sur $1$.

Alors $\Phi(\alpha) : k[A]\to k^{A^\sharp}$ est un isomorphisme.
\end{cor}

\begin{proof}[Démonstration du théorème~\ref{thprinc}] L'anneau $\zcyc[1/p]$ est un module libre (dont une base est formée par les racines de l'unité) sur $\mathbb{Z}[1/p]$. Par conséquent, le corollaire~\ref{cor-tpzc} et le lemme~\ref{lm-cab} montrent que le foncteur d'inclusion $A : \mathbf{Ab}_{(p)}^f\to\mathbf{Ab}$ vérifie la condition $(\mathrm{LD})_{\mathbb{Z}[1/p]}$. Le lemme~\ref{lm-cauni} achève alors la démonstration.
\end{proof}

\subsection{Quelques résultats pour des foncteurs additifs à valeurs infinies}\label{ssect-vi}

\subsubsection{Exemple d'isomorphisme $k[A]\simeq k^B$ où $A$ et $B$ prennent des valeurs infinies}

Pour des foncteurs additifs $A$ et $B$, on peut avoir $k[A]\simeq k^B$ sans que $A$ soit à valeurs finies, comme l'illustre la proposition~\ref{pr-gros_dual} ci-après. Pour l'établir, nous aurons besoin du théorème~\ref{th-Ku} ci-dessous, qui découle des résultats principaux de Kuhn \cite{Ku-adv}.

\begin{nota} Si $\FF$ est un corps fini, on note $\FF-\mathrm{ev}$ la catégorie des $\FF$-espaces vectoriels de dimension finie et l'on abrège $\F(\FF-\mathrm{ev};k)$ en $\F(\FF,k)$.
\end{nota}

\begin{thm}[Kuhn]\label{th-Ku} Soit $\FF$ un corps fini de cardinal $q$ inversible dans $k$. Pour un foncteur $F$ de $\F(\FF,k)$ et un $\FF$-espace vectoriel fini $V$, on définit des représentations $k$-linéaires $F^!(V)$ et $F_!(V)$ du groupe linéaire $\mathrm{Aut}_\FF(V)$ par
$$F^!(V)=\mathrm{Ker}\Big(F(V)\to\bigoplus_D F(V/D)\Big)\quad\text{et}\quad F_!(V)=\mathrm{Coker}\Big(\bigoplus_H F(H)\to F(V)\Big)$$
où $D$ (resp. $H$) parcourt l'ensemble des droites (resp. hyperplans) de $V$. Alors :
\begin{enumerate}
\item les représentations $F^!(V)$ et $F_!(V)$ sont naturellement isomorphes ;
\item si $F$ et $G$ sont des foncteurs de $\F(\FF,k)$ tels que les représentations $F^!(\FF^n)$ et $G^!(\FF^n)$ de $\GL_n(\FF)$ soient isomorphes pour tout $n\in\mathbb{N}$, alors $F$ et $G$ sont isomorphes.
\end{enumerate}
\end{thm}

\begin{pr}\label{pr-gros_dual} Soient $\FF$ un corps fini et $E$ un $\FF$-espace vectoriel (de dimension arbitraire). On suppose que $k$ est un corps  et que $|\FF|\in k^\times$. Alors les foncteurs $P^E:=k[\Hom(E,-)]$ et $I^E:=k^{\Hom(-,E)}$ de $\F(\FF,k)$ sont isomorphes.
\end{pr}

\begin{proof} Il suffit de traiter le cas où la dimension $\mathfrak{c}$ de $E$ est infinie (le cas où $E$ est de dimension finie étant un cas particulier facile du théorème~\ref{thprinc}) ; elle coïncide alors avec le cardinal de $E$, puisque $\FF$ est fini.

Soit $V$ un $\FF$-espace vectoriel fini. On a clairement, avec les notations du théorème~\ref{th-Ku}, $(P^E)_!(V)\simeq k[\surj(E,V)]$, où $\surj(E,V)\subset\Hom(E,V)$ désigne le sous-$\mathrm{Aut}_\FF(V)$-ensemble des applications surjectives. Comme $\surj(E,V)$ est un $\mathrm{Aut}_\FF(V)$-ensemble libre, $(P^E)_!(V)$ est un $k[\mathrm{Aut}_\FF(V)]$-module libre, dont le rang vaut $1$ si $V$ est nul et $2^\mathfrak{c}$ si $V$ est de dimension finie non nulle.

De même, $(I^E)^!(V)\simeq k^{\mathrm{Inj}(V,E)}$, où $\mathrm{Inj}(V,E)\subset\Hom(V,E)$ désigne le sous-$\mathrm{Aut}_\FF(V)$-ensemble des applications injectives, qui est libre. Il s'ensuit que, si $V$ est $\FF$-espace vectoriel de dimension finie non nulle, le $k[\mathrm{Aut}_\FF(V)]$-module  $(I^E)^!(V)$ est isomorphe au produit de $\mathfrak{c}$ copies de $k[\mathrm{Aut}_\FF(V)]$. Un tel module est injectif (comme produit d'injectifs), donc projectif, puisque $\mathrm{Aut}_\FF(V)$ est un groupe fini. Les différents $k[\mathrm{Aut}_\FF(V)]$-modules simples apparaissent chacun avec la multiplicité $2^\mathfrak{c}$ dans son socle, ce qui permet d'en déduire que c'est un module libre de rang $2^\mathfrak{c}$.

On conclut en appliquant le théorème~\ref{th-Ku}.
\end{proof}

\subsubsection{Restrictions liées à la torsion imposées par un isomorphisme $k[A]\simeq k^B$}

Si $k[A]\simeq k^B$, où $A : \A\to\mathbf{Ab}$ et $B : \A^\op\to\mathbf{Ab}$ sont des foncteurs additifs, et que $k$ est de caractéristique nulle, la proposition~\ref{pr-ckt} montre que $A$ est à valeurs dans les groupes abéliens de torsion. Nous allons voir que cette conclusion vaut en fait pour tout anneau $k$ non nul, et que la torsion des valeurs est bornée sur chaque évaluation des foncteurs $A$ et $B$ (proposition~\ref{pr-expbor} ci-après).

Pour le démontrer, nous allons utiliser l'action du monoïde $\mathbb{Z}_\mu$ (ou plus précisément du sous-monoïde des entiers non nuls) sur tout foncteur de $\F(\A;k)$ (qui constitue l'un des principaux outils de \cite{DT-poids} ; nous n'utiliserons ici que des considérations directes et élémentaires). 

\begin{nota} On désigne par $\mathbb{Z}^*$ le monoïde multiplicatif des entiers relatifs non nuls et par $k_\tr$ le $k$-module $k$ muni de l'action triviale de $\mathbb{Z}^*$. 
\end{nota}

\begin{nota} Soient $V$ un groupe abélien et $n\in\mathbb{N}^*$. On note $V_\tor$ le sous-groupe de torsion de $V$ et $_n V$ le sous-groupe des éléments annulés par $n$.
\end{nota}

\begin{lm}\label{lm-torisol} Soient $A : \A\to\mathbf{Ab}$ et $B : \A^\op\to\mathbf{Ab}$ des foncteurs additifs. Tout isomorphisme $k[A]\simeq k^B$ induit des isomorphismes :
\begin{enumerate}
\item\label{it1-tfin} $k[A/n]\simeq k^{\,_n B}$ pour tout $n\in\mathbb{N}^*$ ;
\item\label{it2-tfin} $k[_n A]\simeq k^{B/n}$ pour tout $n\in\mathbb{N}^*$ ;
\item\label{it3-tfin} $k[A_\tor]\simeq\underset{n\in\mathbb{N}^*}{\col} k^{B/n}$ (la colimite étant relative à la divisibilité sur $\mathbb{N}^*$);
\item\label{it4-tfin} $\underset{n\in\mathbb{N}^*}{\lim}k[A/n]\simeq k^{B_\tor}$.
\end{enumerate}
\end{lm}

\begin{proof} Pour $n\in\mathbb{N}^*$, notons $\Phi_n$ (resp. $\Phi^n$) l'endofoncteur de $\F(\A;k)$, muni d'une transformation naturelle canonique $\mathrm{Id}\twoheadrightarrow\Phi_n$ (resp. $\Phi^n\hookrightarrow\mathrm{Id}$), défini par
$$\big(\Phi_n(F)\big)(t):=\mathrm{Coker}\big(F(t\oplus t)\xrightarrow{F(1\quad n)-F(1\quad 0)}F(t)\big)\quad\text{(resp.}$$
$$\big(\Phi^n(F)\big)(t):=\mathrm{Ker}\big(F(t)\xrightarrow{F\left(\begin{array}{c} 1 \\
                                                n 
                                               \end{array}\right)-F\left(\begin{array}{c} 1 \\
                                                0 
                                               \end{array}\right)}F(t\oplus t)\big)\quad\text{).}$$

Un calcul facile fournit des isomorphismes canoniques $\Phi_n(k[A])\simeq k[A/n]$ et $\Phi^n(k[A])\simeq k[_n A]$ (et les morphismes naturels vers ou depuis $k[A]$ sont induits par la projection $A\twoheadrightarrow A/n$ et l'injection $_n A\hookrightarrow A$), et, dualement,  $\Phi_n(k^B)\simeq k^{\,_n B}$ et $\Phi^n(k^B)\simeq k^{B/n}$ (idem). Ainsi, l'assertion \ref{it1-tfin} (resp. \ref{it2-tfin}) s'obtient en appliquant $\Phi_n$ (resp. $\Phi^n$) à l'isomorphisme $k[A]\simeq k^B$, tandis que \ref{it3-tfin} (resp. \ref{it4-tfin}) se déduit de \ref{it2-tfin} (resp. \ref{it1-tfin}) et de la commutation de $k[-] : \mathbf{Ab}\to k\Md$ (resp. $k^{-} : \mathbf{Ab}^\op\to k\Md$) aux colimites (resp. limites) filtrantes.
\end{proof}

\begin{rem} La démonstration précédente utilise dans un cas particulier simple la notion de \textit{radical} d'un foncteur introduite dans \cite[§\,6.1]{DT-poids}.
\end{rem}

%[Ordre des lemmes à voir.]

\begin{lm}\label{lm-tor_Zmu1} Soit $V$ un groupe abélien. L'inclusion $V_\tor\hookrightarrow V$ induit un isomorphisme $\Hom_{k[\mathbb{Z^*}]}(k_\tr,k[V_\tor])\xrightarrow{\simeq}\Hom_{k[\mathbb{Z}^*]}(k_\tr,k[V])$.
\end{lm}

\begin{proof} Soit $\sum_{i=1}^r\lambda_i[v_i]$, où les $\lambda_i$ sont des éléments non nuls de $k$ et les $v_i$ des éléments deux à deux distincts de $V$, un élément de $k[V]$ invariant par l'action de $\mathbb{Z}^*$ : pour tout $n\in\mathbb{Z}^*$, on a $\sum_{i=1}^r\lambda_i[v_i]=\sum_{i=1}^r\lambda_i[nv_i]$, ce qui entraîne que la multiplication par $n$ induit une permutation de l'ensemble $\{v_1,\dots,v_r\}$. En particulier, pour tout $i$, l'ensemble $\mathbb{Z}^*.v_i$ est inclus dans $\{v_1,\dots,v_r\}$, donc fini, ce qui implique que $v_i$ est de torsion et termine la démonstration.
%Cela résulte de ce que l'orbite d'un élément de $V$ qui n'est pas torsion pour l'action du monoïde $\mathbb{Z}^*$ est infinie. [UN PEU ELLIPTIQUE ?]
\end{proof}

\begin{lm}\label{lm-tor_Zmu2} Soit $U$ un groupe abélien. Si l'inclusion $\underset{n\in\mathbb{N}^*}{\col} k^{U/n}\hookrightarrow k^U$ (la colimite étant relative à la divisibilité sur $\mathbb{N}^*$) induit un isomorphisme\linebreak[4] $\Hom_{k[\mathbb{Z}^*]}(k_\tr,\underset{n\in\mathbb{N}^*}{\col} k^{U/n})\xrightarrow{\simeq}\Hom_{k[\mathbb{Z}^*]}(k_\tr,k^U)$, alors $U$ est un groupe abélien de torsion.
\end{lm}

\begin{proof} Considérons la fonction $\alpha : U\to k$ associant $0$ aux éléments de torsion de $U$ et $1$ aux autres éléments. Pour $n\in\mathbb{Z}^*$ et $u\in U$, on a : $(n.u\in U_\tor)\Leftrightarrow (u\in U_\tor)$, de sorte que $\alpha$ définit une application $k[\mathbb{Z}^*]$-linéaire $k_\tr\to k^U$. Si cette application appartient à l'image du monomorphisme canonique $\Hom_{k[\mathbb{Z}^*]}(k_\tr,\underset{n\in\mathbb{N}^*}{\col} k^{U/n})\to\Hom_{k[\mathbb{Z}^*]}(k_\tr,k^U)$, alors il existe $n\in\mathbb{N}^*$ tel que $\alpha$ se factorise par la réduction modulo $n$. En particulier, $\alpha$ envoie tout élément de $n.U$ sur $\alpha(0)=0$, d'où $n.U\subset U_\tor$, ce qui n'est possible que si $U$ est un groupe de torsion.
\end{proof}

\begin{lm}\label{lm-limlin} Soient $V$ un groupe abélien et $(U_n)_{n\in\mathbb{N}}$ une suite décroissante de sous-groupes de $V$. Si le morphisme canonique $k[V]\to\underset{n\in\mathbb{N}}{\lim}\,k[V/U_n]$ est un isomorphisme, alors il existe $n\in\mathbb{N}$ tel que $U_n=0$.
\end{lm}

\begin{proof} L'injectivité du morphisme canonique implique clairement la nullité de $\underset{n\in\mathbb{N}}{\bigcap}U_n$. Si tous les $U_n$ sont non nuls, quitte à extraire une sous-suite, on peut supposer que chaque inclusion $U_{n+1}\subset U_n$ est stricte. Choisissons donc $u_n\in U_n\setminus U_{n+1}$. La suite $(\sum_{i=0}^{n-1}([u_i\mod U_n]-[0]))_{n\in\mathbb{N}}$ de $\prod_{n\in\mathbb{N}}k[V/U_n]$ définit un élément de $\underset{n\in\mathbb{N}}{\lim}\,k[V/U_n]$. Pour montrer qu'il ne peut pas appartenir à l'image du morphisme canonique, introduisons une notation : pour un groupe abélien $T$ et un élément $\xi$ de $k[T]$, notons $\lgr(\xi)$ le plus petit $n\in\mathbb{N}$ tel qu'il existe des éléments $t_i$ de $V$ et $\lambda_i$ de $k$, pour $1\le i\le n$, vérifiant $\xi=\sum_{i=1}^n\lambda_i[t_i]$. Si les $\xi_n$ (pour $n\in\mathbb{N}$) sont des éléments de $k[V/U_n]$ tels que $(\xi_n)$ appartienne à l'image du morphisme canonique $k[V]\to\underset{n\in\mathbb{N}}{\lim}\,k[V/U_n]\hookrightarrow\prod_{n\in\mathbb{N}}k[V/U_n]$, alors la suite $(\lgr(\xi_n))_{n\in\mathbb{N}}$ est bornée.

Or $\lgr\big(\sum_{i=0}^{n-1}([u_i\mod U_n]-[0])\big)\ge n$, car les classes dans $V/U_n$ de $0$ et des $u_i$, pour $0\le i\le n-1$, sont deux à deux distinctes (si $j<i<n$, alors $u_i-u_j\notin U_i$ car $u_i\in U_i$ et $u_j\notin U_i$, en particulier $u_i-u_j\notin U_n$). Cela termine la démonstration.
\end{proof}

\begin{pr}\label{pr-expbor} Soient $A : \A\to\mathbf{Ab}$ et $B : \A^\op\to\mathbf{Ab}$ des foncteurs additifs tels que $k[A]\simeq k^B$ dans $\F(\A;k)$. Alors pour tout objet $x$ de $\A$, il existe $n\in\mathbb{N}^*$ qui annule les groupes abéliens $A(x)$ et $B(x)$.
\end{pr}

\begin{proof} Les lemmes~\ref{lm-tor_Zmu1}, \ref{lm-tor_Zmu2} et \ref{lm-torisol}.\,\ref{it3-tfin} montrent tout d'abord que $B$ est à valeurs dans les groupes abéliens de torsion.

Le lemme~\ref{lm-torisol}.\,\ref{it4-tfin} montre alors que le morphisme canonique
$$k[A]\to\underset{n\in\mathbb{N}^*}{\lim}k[A/n]\simeq\underset{m\in\mathbb{N}}{\lim}k[A/m!]$$
est un isomorphisme. Le lemme~\ref{lm-limlin} permet d'en déduire que pour tout objet $x$ de $\A$, il existe $n\in\mathbb{N}^*$ tel que $n.A(x)=0$, d'où également $n.B(x)=0$ grâce au lemme~\ref{lm-torisol}.\,\ref{it1-tfin}.
\end{proof}

\begin{rem} La condition de torsion sur le foncteur additif $B : \A^\op\to\mathbf{Ab}$ de la proposition~\ref{pr-expbor} n'est pas suffisante pour qu'il existe un foncteur additif $A : \A\to\mathbf{Ab}$ tel que $k[A]\simeq k^B$, et ce quel que soit l'anneau (non nul) $k$.

Pour le voir, prenons par exemple pour $\A$ la catégorie opposée de la catégorie des $\mathbb{Z}/p$-espaces vectoriels au plus dénombrables et pour $B$ le foncteur d'inclusion. Supposons qu'existe un foncteur additif $A : \A\to\mathbf{Ab}$ tel que $k[A]\simeq k^B$. Considérons le foncteur $\mathbb{N}\to\A^\op\quad n\mapsto (\mathbb{Z}/p)^n$, où $\mathbb{N}$ est vu comme la catégorie associée à l'ordre usuel sur les nombres naturels, chaque flèche $n\to n+1$ étant envoyée sur l'inclusion $(\mathbb{Z}/p)^n\to (\mathbb{Z}/p)^{n+1}$ des $n$ premiers facteurs. Comme $k^B$ commute aux limites filtrantes dénombrables, l'application linéaire canonique
$$k\big[A\big((\mathbb{Z}/p)^{\oplus\mathbb{N}}\big)\big]\to\underset{n\in\mathbb{N}}{\lim}\, k\big[A\big((\mathbb{Z}/p)^n\big)\big]$$
est un isomorphisme.

Par ailleurs, du fait que les morphismes canoniques $(\mathbb{Z}/p)^n\to(\mathbb{Z}/p)^{\oplus\mathbb{N}}$ sont des monomorphismes \textit{scindés} de $\A^\op$, ils induisent des épimorphismes (scindés) de groupes abéliens $A\big((\mathbb{Z}/p)^{\oplus\mathbb{N}}\big)\twoheadrightarrow A\big((\mathbb{Z}/p)^n\big)$. Le lemme~\ref{lm-limlin} permet d'en déduire que ces épimorphismes sont des isomorphismes pour $n$ assez grand, i.e. que $k^{(\mathbb{Z}/p)^{\oplus\mathbb{N}}}\twoheadrightarrow k^{(\mathbb{Z}/p)^n}$ est un isomorphisme pour $n$ assez grand, absurdité qui conclut notre exemple (qui est à mettre en regard de la proposition~\ref{pr-gros_dual}).
\end{rem}

\paragraph*{Remerciements} L'auteur témoignage sa gratitude envers Benachir El Allaoui pour lui avoir posé la question donnant lieu au résultat principal de ce texte. Il remercie égalament Antoine Touzé et Christine Vespa pour des discussions reliées à ce travail.

%%%%%%%%%%%%%%%%%
%On note $\qcyc$ (resp. $\zcyc$) le sous-corps (resp. sous-anneau) de $\mathbb{C}$ engendré par les racines de l'unité. Ces dernières forment une \textit{base} du $\mathbb{Z}$-module (resp. $\mathbb{Q}$-espace vectoriel) $\zcyc$ (resp. $\qcyc$), et $\zcyc$ est l'anneau des entiers de $\qcyc$... et peut-être introduire les variantes "finies", et avec inversions d'entiers utiles ?
%
%A PRIORI PAS BESOIN DE $\qcyc$ !
%%%%%%%%%%%%%%%%%

\bibliographystyle{plain}
\bibliography{bib-linearis.bib}
 
 \end{document}